\documentclass[11pt]{article}
\UseRawInputEncoding
\usepackage{amsmath,amsthm,amssymb,amsfonts}
\usepackage{float}
\usepackage{color,fullpage}
\usepackage{cite}
\usepackage[title,titletoc]{appendix}
\usepackage{algorithm,algorithmic}
\usepackage{tikz}
\usepackage[overload]{empheq}

\usepackage{amsthm}
\newcommand{\RR}{\mathbb{R}}

\newcommand{\cX}{{\mathcal{X}}}

\newtheorem{lemma}{Lemma}[section]
\newtheorem{definition}{Definition}[section]

\newtheorem{corollary}{Corollary}[section]

\newtheorem{proposition}{Proposition}[section]

\makeatletter

\@addtoreset{equation}{section} \makeatother

\begin{document}

\title{New insights in smoothness and strong convexity with improved convergence of gradient descent}

\author{
Lu Zhang\thanks{
Department of Mathematics, National University of Defense Technology,
Changsha, Hunan 410073, China}
\and Jiani Wang\thanks{
School of Mathematical Sciences,
Chinese Academy of Sciences, Beijing 100049, China
}
\and Hui Zhang\thanks{Corresponding author.
Department of Mathematics, National University of Defense Technology,
Changsha, Hunan 410073, China.  Email: \texttt{h.zhang1984@163.com}
}
}
\date{\today}

\maketitle

\begin{abstract}
The starting assumptions to study the convergence and complexity of gradient-type methods may be the smoothness (also called Lipschitz continuity of gradient) and the strong convexity.
In this note, we revisit these two basic properties from a new perspective that motivates their definitions and equivalent characterizations, along with an improved linear convergence of the gradient descent method.
\end{abstract}




\section{Introduction}
In the field of optimization, certain classes of functions have to be figured out so that useful optimization theory (including optimality, complexity, and convergence, etc.) can be established. The class of functions with Lipschitz continuous gradient is one of the most important classes, for which some convergence theory of the gradient descent method can be built. The gradient Lipschitz continuity for an objective function $f$, usually called smoothness for brevity, is tantamount to requiring that the gradient of $f$ is continuous in the following sense:
\begin{equation}\label{lipg}
  \|\nabla f(x)-\nabla f(y)\|\leq L\|x-y\|,
\end{equation}
where $L>0$ is called the Lipschitz constant. However, this class of functions is too wide to guarantee convergence even to a local minimum, let alone get global performance. Thus, the strong convexity was imposed as a restriction so that better convergence guarantees are possible. It is defined by the following property:
\begin{equation}\label{scg}
 f(y)\geq f(x)+\langle \nabla f(x), y-x \rangle +\frac{1}{2}\mu\|x-y\|^2,
\end{equation}
where $\mu>0$ is called the strong-convexity constant.

Up to now, it has become a well-known fact that both smoothness and strong convexity are very fundamental properties for analyzing minimization processes, especially for gradient-type methods. Many existing classic textbooks stated them as basic definitions, along with a couple of equivalent characterizations; see e.g. \cite{nesterov2004introductory,2014Convex,beck2017first}. In the literature of optimization, they were widely cited as the most natural assumptions. However, it seems unclear where they come from, why they have their own forms \eqref{lipg} and \eqref{scg}, and how the existing equivalent conditions be constructed. In this short note, we try to answer these questions based on a simple observation in gradient descent, that is, each subproblem in the gradient descent method should be a ``good'' approximation to the objective function. From this perspective, we define the smoothness and strong convexity as a quantified description of the ``good" approximation. In order to construct a series of equivalent characterizations for these two properties, we first reformulate our definitions into convexity of certain functions and then apply a group of equivalent convexity conditions to them.

The remainder of the paper is organized as follows. In Section \ref{semotive}, we introduce the simple observation in gradient descent, and then define the smoothness and strong-convexity. In Section \ref{seequ}, we present a group of characterizations for the smoothness and strong-convexity. In Section \ref{segd}, we obtain an improved linear convergence by using one of the equivalent characterizations of the convex $L$-smoothness.

\textbf{Notation.} We restrict our attention to an arbitrary finite dimensional space $\RR^d$ associated with dot product $\langle x, y\rangle:=\sum_{i=1}^dx_iy_i$ and its induced norm $\|\cdot\|:=\sqrt{\langle \cdot, \cdot\rangle}$. For a fixed function $f$, the Bregman distance is defined as $$D_f(x,y):=f(x)-f(y)-\langle \nabla f(y),x-y\rangle,$$
the conjugate function $f^*$ is given by $$f^*(u)=\sup_{x\in \RR^d}\{\langle x,u \rangle -f(x)\}.$$

\section{Motivation of definition}\label{semotive}
We begin our study with the following simple observation in gradient descent, which will motivate us to define more general notions than the smoothness and strong-convexity.
\subsection{A simple observation from gradient descent}
The gradient descent is a well-known method for solving $\min_{x\in\RR^d} f(x)$; for $k\geq 0$ it reads as
\begin{equation}\label{GD}
 x_{k+1}=x_k-\frac{1}{t}\nabla f(x_k),
\end{equation}
where $t>0$ is the step size. From the perspective of approximation, the gradient descent scheme \eqref{GD} can be obtained by solving a quadratic approximation subproblem as follows:
\begin{equation}\label{GD1}
x_{k+1}:=\arg\min_{x\in \RR^d}\{f(x_k)+\langle \nabla f(x_k),x-x_k\rangle +\frac{t}{2}\|x-x_k\|^2\}.
\end{equation}
In order to force $x_{k+1}$ to be a good approximation to the minimizer of $f$, we expect that the objective function in the subproblem \eqref{GD1} is a ``good" approximation to the original function $f$, that is
$$f(x)\approx f(x_k)+\langle \nabla f(x_k),x-x_k\rangle +\frac{t}{2}\|x-x_k\|^2.$$
Put the terms about $f$ together to yield the following equivalent approximation formulation
\begin{equation}\label{app}
\frac{t}{2}\|x-x_k\|^2\approx f(x)- f(x_k)-\langle \nabla f(x_k),x-x_k\rangle.
\end{equation}
To write the formulation above in a symmetrical way, we let $\widetilde{\varphi}(x)=\frac{t}{2}\|x\|^2$. Then, it becomes
\begin{equation}\label{app1}
\widetilde{\varphi}(x)- \widetilde{\varphi}(x_k)-\langle \nabla \widetilde{\varphi}(x_k),x-x_k\rangle\approx f(x)- f(x_k)-\langle \nabla f(x_k),x-x_k\rangle.
\end{equation}
Since there may exist some more \textsl{suitable} function $\varphi$ than $\widetilde{\varphi}$ such that the approximation \eqref{app1} becomes better, we consider the following more general approximation:
\begin{equation}\label{app2}
\varphi(x)- \varphi(x_k)-\langle \nabla \varphi(x_k),x-x_k\rangle\approx f(x)- f(x_k)-\langle \nabla f(x_k),x-x_k\rangle.
\end{equation}
In terms of the Bregman distance, \eqref{app2} can be simplified into
\begin{equation}\label{app3}
D_\varphi(x,x_k)\approx D_f(x,x_k).
\end{equation}
Due to the arbitrariness of the initial point $x_0$, we actually require that
\begin{equation}\label{app4}
D_\varphi(x,y)\approx D_f(x,y),~ \forall ~x, y\in \RR^d.
\end{equation}

There are at least two ways to quantify the approximation \eqref{app4}. The first way is to bound the difference between $D_\varphi(x,y)$ and $D_f(x,y)$, that is there exists a constant $C< +\infty$ such that
$$\sup_{x,y\in\RR^d}|D_\varphi(x,y)-D_f(x,y)|= C.$$
A drawback is that the constant $C$ is not scale invariant since for any $\gamma>0$, we have
$$\sup_{x,y\in\RR^d}|D_{\gamma\varphi}(x,y)-D_{\gamma f}(x,y)|= \gamma C.$$
To overcome it, we consider the second way that bound the ratio between them, that is there exists two constants $\mu$ and $L$ such that for any $x, y\in \RR^d$ with $x\neq y$ we have
\begin{equation}\label{app5}
\mu\leq \frac{D_f(x,y)}{D_\varphi(x,y)}\leq L.
\end{equation}
Further, we assume the strictly convexity of $\varphi$ so that the above can be equivalently written as
\begin{equation}\label{app6}
\mu D_\varphi(x,y)\leq D_f(x,y)\leq L D_\varphi(x,y), \forall x,y\in\RR^d,
\end{equation}
which is just the relative smoothness (also called Lipschitz-like/convexity condition in \cite{Bauschke2016A}) and the relatively strong convexity, proposed recently in \cite{Lu2016relatively}.
Since the purpose of this note is to study the smoothness and the strong-convexity, we restrict our attention to \eqref{app6} with $\varphi=\varphi_0:=\frac{1}{2}\|\cdot\|^2$. The pursuit of more general discussion may be leaved as future work.

\subsection{The smoothness and strong-convexity}
Now, we are ready to define the smoothness and the strong-convexity in terms of the Bregman distance as follows.
\begin{definition}\label{def1}
Let $f:\RR^d\rightarrow \RR$ be a differentiable function and $\varphi_0=\frac{1}{2}\|\cdot\|^2$. Let $L$ and $\mu$ be two given nonnegative constant. We say that $f$ is $L$-smooth if
\begin{equation}\label{app7}
 D_f(x,y)\leq L D_{\varphi_0}(x,y), \forall x,y\in\RR^d,
\end{equation}
and $\mu$-strongly-convex if
\begin{equation}\label{app8}
\mu D_{\varphi_0}(x,y)\leq D_f(x,y), \forall x,y\in\RR^d.
\end{equation}
\end{definition}

After some simple calculations, we can separately write the $L$-smoothness and the $\mu$-strong-convexity in the following equivalent forms
\begin{subequations}\label{app9}
\begin{align}
\label{sm}&f(y)\leq f(x)+\langle \nabla f(x), y-x\rangle+\frac{L}{2}\|y-x\|^2,\\
\label{sc}&f(y)\geq f(x)+\langle \nabla f(x), y-x\rangle+\frac{\mu}{2}\|y-x\|^2.
\end{align}
\end{subequations}
They are the usual definitions appeared in textbooks. The advantage of Definition \ref{def1} is that we can easily turn \eqref{sm}-\eqref{sc} into convexity of certain functions. Actually, applying the linearity of the Bregman distance to \eqref{app7}, we obtain $D_{L\varphi_0-f}(y,x)\geq0$, which equivalently means the convexity of $L\varphi_0-f$. Similarly, the $\mu$-strong-convexity is equivalent to the convexity of $f-\mu\varphi_0$. We summarize these results in the following lemma.

\begin{lemma}\label{lemdef}
Let $f:\RR^d\rightarrow \RR$ be a differentiable function and $\varphi_0=\frac{1}{2}\|\cdot\|^2$. Then,
\begin{enumerate}
  \item[1).] $f$ is $L$-smooth if and only if $L\varphi_0-f$ is convex;
  \item[2).] $f$ is $\mu$-strongly-convex if and only if $f-\mu\varphi_0$ is convex.
\end{enumerate}
\end{lemma}

\section{Characterization of equivalence}\label{seequ}
In this section, we will present three groups of equivalent characterizations for the smoothness and strong convexity. The first group is based on Lemma \ref{lemdef} and different definitions of convexity.  The second group is obtained from a point of dual view. The last group is obtained by using some Fenchel dual properties. To this end, we first introduce two basic results.

\begin{lemma}[Proposition 14.2 in \cite{Bauschke2017}]\label{lemdual}
Let $f:\RR^d\rightarrow \RR$ be a differentiable convex function, $\varphi_0=\frac{1}{2}\|\cdot\|^2$, and $\gamma>0$. Then, $\gamma\varphi_0-f$ is convex if and only if $f^*-\gamma^{-1}\varphi_0$ is convex.
\end{lemma}

\begin{lemma}[Theorem 23.5 and Corollary 23.5.1 in \cite{1970convex}]\label{lemfenc}
Let $f:\RR^d\rightarrow \RR$ be a convex function. Then, the conditions $f(u)+f^*(u^*)=\langle u, u^*\rangle$, $u^*\in \partial f(u)$, and $u\in\partial f^*(u^*)$ are equivalent. Moreover, $\partial f^*$ is the inverse of $\partial f$.
\end{lemma}
We also need the equivalent definitions of convexity.
\begin{lemma}[Theorem 2.14 in \cite{Rockafellar2004Variational}]\label{lemconv}
Let $f:\RR^d\rightarrow \RR$ be a differentiable function. Then, it is convex if and only if one of the following conditions holds:
\begin{subequations}\label{defconv}
\begin{align}
\label{conv1}& f(\lambda x+(1-\lambda)y)\leq \lambda f(x)+(1-\lambda)f(y),\forall x,y\in\RR^d, \lambda\in [0,1],\\
\label{conv2}& \langle \nabla f(x)-\nabla f(y), x-y\rangle \geq 0, \forall x,y\in\RR^d,\\
\label{conv3}& f(y)\geq f(x)+\langle \nabla f(x), y-x\rangle, \forall x,y\in\RR^d.
\end{align}
\end{subequations}
\end{lemma}

\subsection{Characterizations of smoothness}
Combining the first statement in Lemma \ref{lemdef} with the different convexity conditions in Lemma \ref{lemconv}, we have the following equivalent characterization of smoothness.
\begin{proposition}\label{pro1}
Let $f:\RR^d\rightarrow \RR$ be a differentiable function. Then, it is $L$-smooth if and only if one of the following conditions holds:
\begin{subequations}\label{smooth1}
\begin{align}
\label{sm1}& f(\lambda x+(1-\lambda)y)\geq \lambda f(x)+(1-\lambda)f(y)-\frac{L}{2}\lambda(1-\lambda)\|x-y\|^2,\forall x,y\in\RR^d, \lambda\in [0,1],\\
\label{sm2}& \langle \nabla f(x)-\nabla f(y), x-y\rangle \leq L\|x-y\|^2, \forall x,y\in\RR^d,\\
\label{sm3}& f(y)\leq f(x)+\langle \nabla f(x), y-x\rangle+\frac{L}{2}\|x-y\|^2, \forall x,y\in\RR^d.
\end{align}
\end{subequations}
\end{proposition}

In order to obtain a group of dual characterizations, we need to assume that the differentiable function $f$ is also convex so that Lemma \ref{lemdual} can be invoked. In other words, the convexity of $L\varphi_0-f$ is equivalent to the convexity of $f^*-L^{-1}\varphi$. Before stating the dual characterizations, we denote by $\Gamma(\RR^d)$ the class of functions with the following properties:
\begin{itemize}
  \item $f$ is convex and differentiable;
  \item its conjugate $f^*$ is differentiable.
\end{itemize}

Now, together with Lemma \ref{lemconv}, we immediately have the following result, which seems new to the best of our knowledge. It is obtained at the cost of restricting the objectives into the class of functions $\Gamma(\RR^d)$.
\begin{proposition}\label{pro2}
Let $f\in \Gamma(\RR^d)$. Then, it is $L$-smooth if and only if one of the following conditions holds:
\begin{subequations}\label{smooth2}
\begin{align}
\label{dsm1}& f^*(\lambda u+(1-\lambda)v)\leq \lambda f^*(u)+(1-\lambda)f^*(v)-\frac{1}{2L}\lambda(1-\lambda)\|u-v\|^2,\forall u,v\in\RR^d, \lambda\in [0,1],\\
\label{dsm2}& \langle \nabla f^*(u)-\nabla f^*(v), u-v\rangle \geq \frac{1}{L}\|u-v\|^2, \forall u,v\in\RR^d,\\
\label{dsm3}& f^*(u)\geq f^*(v)+\langle \nabla f^*(v), u-v\rangle+\frac{1}{2L}\|u-v\|^2, \forall u,v\in\RR^d.
\end{align}
\end{subequations}
\end{proposition}

\begin{corollary}\label{cor1}
Let $f$ be convex and differentiable. Then, it is $L$-smooth if and only if one of the following conditions holds:
\begin{subequations}\label{smooth3}
\begin{align}
\label{psm1}& \langle \nabla f(x)-\nabla f(y), x-y\rangle \geq \frac{1}{L}\|\nabla f(x)-\nabla f(y)\|^2, \forall x,y\in\RR^d,\\
\label{psm2}& f(y)\geq f(x)+\langle \nabla f(x), y-x\rangle+\frac{1}{2L}\|\nabla f(x)-\nabla f(y)\|^2, \forall x,y\in\RR^d,\\
\label{psm3}& \|\nabla f(x)-\nabla f(y)\|\leq L\|x-y\|, \forall x,y\in\RR^d.
\end{align}
\end{subequations}
\end{corollary}
\begin{proof}
First of all, we show the equivalence under the assumption of $f\in \Gamma(\RR^d)$ so that Proposition \ref{pro2} can be applied. Actually, the first two conditions follows from \eqref{dsm2} and \eqref{dsm3} separately by letting $\nabla f^*(u)=x, \nabla f^*(v)=y$ and using the fact $(\nabla f )^{-1}=\nabla f^*$ and the following relationships in Lemma \ref{lemfenc}:
\begin{subequations}
\begin{align}
 & f^*(u)+f(x)=\langle u,x\rangle=\langle \nabla f(x),x\rangle,\\
 & f^*(v)+f(y)=\langle v,y\rangle=\langle \nabla f(y),y\rangle.
\end{align}
\end{subequations}
$L$-smoothness $\Rightarrow$ \eqref{psm3} follows by applying the Cauchy-Schwartz inequality to \eqref{psm1}; while \eqref{psm3}  $\Rightarrow$ $L$-smoothness follows by applying the Cauchy-Schwartz inequality to $\langle \nabla f(x)-\nabla f(y), x-y\rangle$ to obtain \eqref{sm2}.

Now, we drop the restriction of $f^*$ being differentiable. To this end, we let $f_\epsilon:=f+\epsilon \varphi_0$ with $\epsilon>0$, i.e., $f_\epsilon$ must be strongly convex and hence its conjugate is differentiable \cite{1970convex}. Therefore, $f_\epsilon\in \Gamma(\RR^d)$. Now, we complete the proof by showing the following chain of implication:
$$\eqref{sm}\Rightarrow \eqref{psm2}\Rightarrow \eqref{psm1}\Rightarrow\eqref{psm3} \Rightarrow\eqref{sm}.$$
In fact, \eqref{sm} says that $f$ is $L$-smooth and hence $L\varphi_0-f$ is convex. This further implies that $(L+\epsilon)\varphi_0-f_\epsilon$ is convex, i.e., $f_\epsilon$ is $(L+\epsilon)$-smooth. Thus, applying \eqref{psm2} to $f_\epsilon$, we have
$$f_\epsilon(y)\geq f_\epsilon(x)+\langle \nabla f_\epsilon(x), y-x\rangle+\frac{1}{2(L+\epsilon)}\|\nabla f_\epsilon(x)-\nabla f_\epsilon(y)\|^2, \forall x,y\in\RR^d.$$
Letting $\epsilon$ in the above inequality tend to zero, we immediately obtain \eqref{psm2}. From \eqref{psm2}, we have  $$ f(y)\geq f(x)+\langle \nabla f(x), y-x\rangle+\frac{1}{2L}\|\nabla f(x)-\nabla f(y)\|^2,$$
$$ f(x)\geq f(y)+\langle \nabla f(y), x-y\rangle+\frac{1}{2L}\|\nabla f(x)-\nabla f(y)\|^2.$$
Adding these two inequalities, we get \eqref{psm1}, which implies \eqref{psm3} by invoking the Cauchy-Schwartz inequality. It remains to show $\eqref{psm3} \Rightarrow\eqref{sm}$. By the fundamental theorem of calculus, we have
$$f(y)-f(x)=\int_0^1\langle \nabla f(x+t(y-x)), y-x\rangle dt.$$
Therefore,
\begin{eqnarray*}
\begin{array}{lll}
f(y)-f(x)-\langle \nabla f(x),y-x\rangle  &= \int_0^1\langle \nabla f(x+t(y-x))-\nabla f(x), y-x\rangle dt\\
&\leq \int_0^1\|\nabla f(x+t(y-x))-\nabla f(x)\| \|y-x\|dt\\
&\leq \int_0^1 tL \|y-x\|^2dt=\frac{L}{2}\|x-y\|^2,
\end{array}
\end{eqnarray*}
where the first inequality follows by the Cauchy-Schwartz inequality and the second one from \eqref{psm3}. This completes the proof.
\end{proof}
Combining Proposition \ref{pro1} and Corollary \ref{cor1}, we conclude that for any convex and differentiable function $f$, its $L$-smoothness can be equivalently characterized by one of the conditions \eqref{sm1}-\eqref{sm3} and \eqref{psm1}-\eqref{psm3}. Especial attention should be paid to the conditions \eqref{psm1} and \eqref{psm2}, since each of them implies the convexity of $f$. Thus, we have the following additional result.
\begin{corollary}\label{cor2}
Let $f$ be differentiable. Then, it is convex and $L$-smooth if and only if one of the conditions \eqref{psm1} and \eqref{psm2} holds.
\end{corollary}

It should be noted although the equivalence between the $L$-smoothness and the conditions \eqref{psm1} and \eqref{psm2} is well-known for convex and differentiable, the current presentation in Corollary \eqref{cor2} seems more accurate.

\subsection{Characterizations of strong-convexity}
The results in this part will be obtained in the similar way to that in the previous subsection. Combining the second statement in Lemma \ref{lemdef} with the different convexity conditions in Lemma \ref{lemconv}, we have the following equivalent characterization of strong-convexity.
\begin{proposition}
Let $f:\RR^d\rightarrow \RR$ be a differentiable function. Then, it is $\mu$-stongly-convex if and only if one of the following conditions holds:
\begin{subequations}\label{sconv1}
\begin{align}
\label{sc1}& f(\lambda x+(1-\lambda)y)\leq \lambda f(x)+(1-\lambda)f(y)-\frac{\mu}{2}\lambda(1-\lambda)\|x-y\|^2,\forall x,y\in\RR^d, \lambda\in [0,1],\\
\label{sc2}& \langle \nabla f(x)-\nabla f(y), x-y\rangle \geq \mu\|x-y\|^2, \forall x,y\in\RR^d,\\
\label{sc3}& f(y)\geq f(x)+\langle \nabla f(x), y-x\rangle+\frac{\mu}{2}\|x-y\|^2, \forall x,y\in\RR^d.
\end{align}
\end{subequations}
\end{proposition}
The corresponding dual characterizations are followings:
\begin{proposition}\label{strongconvex}
Let $f:\RR^d\rightarrow \RR$ be a differentiable function. Then, it is $\mu$-stongly-convex if and only if one of the following conditions holds:
\begin{subequations}\label{sconv2}
\begin{align}
\label{dsc1}& f^*(\lambda u+(1-\lambda)v)\geq \lambda f^*(u)+(1-\lambda)f^*(v)-\frac{1}{2\mu}\lambda(1-\lambda)\|u-v\|^2,\forall u,v\in\RR^d, \lambda\in [0,1],\\
\label{dsc2}& \langle \nabla f^*(u)-\nabla f^*(v), u-v\rangle \leq \frac{1}{\mu}\|u-v\|^2, \forall u,v\in\RR^d,\\
\label{dsc3}& f^*(u)\leq f^*(v)+\langle \nabla f^*(v), u-v\rangle+\frac{1}{2\mu}\|u-v\|^2, \forall u,v\in\RR^d.
\end{align}
\end{subequations}
\end{proposition}

The above clarifies the equivalent conditions between $\mu$-stong-convexity of $f$ and $\mu^{-1}$-smoothness of $f^*$. A classic example is $f(x)=\frac{1}{2}x^TQx$, where $Q$ is a non-singular symmetric positive semi-definite $d\times d$ matrix; the conjugate $f^*(x)=\frac{1}{2}x^TQ^{-1}x$. Obviously, $f$ is $\lambda_{\min}(Q)$-stongly-convex while $f^*$ is $\lambda_{\min}(Q)^{-1}$-smooth, where $\lambda_{\min}(Q)$ represents the smallest eigenvalue of $Q$.

\begin{corollary}
Let $f:\RR^d\rightarrow \RR$ be a differentiable function. Then, it is $\mu$-stongly-convex if and only if one of the following conditions holds:
\begin{subequations}\label{sconv3}
\begin{align}
\label{psc1}& \langle \nabla f(x)-\nabla f(y), x-y\rangle \leq \frac{1}{\mu}\|\nabla f(x)-\nabla f(y)\|^2, \forall x,y\in\RR^d,\\
\label{psc2}& f(y)\leq f(x)+\langle \nabla f(x), y-x\rangle+\frac{1}{2\mu}\|\nabla f(x)-\nabla f(y)\|^2, \forall x,y\in\RR^d,\\
\label{psc3}& \|\nabla f(x)-\nabla f(y)\|\geq \mu\|x-y\|, \forall x,y\in\RR^d.
\end{align}
\end{subequations}
\end{corollary}

The proof is similar to that of Corollary \ref{cor1}; we omit the details.

\subsection{Characterizations of smoothness and strong-convexity}
Now, we want to characterize the class of functions that satisfy \eqref{sm} and \eqref{sc}, i.e., both $f-\mu \varphi_0$ and $L\varphi_0-f$ are convex. Let $\tilde{f}:= f-\mu \varphi_0$. Then, it is equivalent to the condition that $\tilde{f}$ is convex and $(L-\mu)$-smooth. Recalling the equivalent conditions in Corollary \ref{cor2}, we have the following characterizations for smoothness and strong-convexity.
\begin{proposition}
Let $f:\RR^d\rightarrow \RR$ be a differentiable function. Then, it is $L$-smooth and $\mu$-strongly-convex if and only if one of the following conditions holds:
\begin{subequations}\label{sm-sc}
\begin{align}
\label{sm-sc1}\langle \nabla f(x)-\nabla f(y), x-y\rangle \leq& \frac{L\mu}{L+\mu}\|x-y\|^2+\frac{1}{L+\mu}\|\nabla f(x)-\nabla f(y)\|^2, \forall x,y\in\RR^d,\\
\label{sm-sc2} f(y)\geq& f(x)+\langle \nabla f(x), y-x\rangle+\frac{1}{2L} \|\nabla f(x)-\nabla f(y)\|^2+\\ \nonumber
 &\frac{\mu L}{2(L-\mu)}\|x-y-\frac{1}{L}(\nabla f(x)-\nabla f(y))\|^2,
 \forall u,v\in\RR^d.
\end{align}
\end{subequations}
\end{proposition}
The equivalent condition \eqref{sm-sc2} was essentially discovered in \cite{Taylor2016smooth}. As a necessary condition to the $L$-smoothness and $\mu$-strong-convexity, \eqref{sm-sc1} appeared in the textbook \cite{nesterov2004introductory}. Here, we highlight that it is also sufficient for a differentiable $f$ to be $L$-smooth and $\mu$-strongly-convex.

\section{Improved linear convergence for gradient descent}\label{segd}
First of all, we recall a well-known linear convergence result for the gradient descent method, stated in \cite{Karimil2016linear}.
\begin{lemma}\label{lincon}
Consider the unconstrained optimization problem $\min_{x\in\RR^d} f(x)$, where $f$ is $L$-smooth, has a nonempty solution set $\cX^*$, and satisfies the Polyak-{\L}ojasiewicz(PL) inequality
\begin{equation}\label{pl}
\frac{1}{2}\|\nabla f(x)\|^2\geq \nu (f(x)-\bar{f}), \forall x\in \RR^d,
\end{equation}
where $\bar{f}$ denotes the optimal function value. Then the gradient descent \eqref{GD} with a step-size of $\frac{1}{L}$ has a global linear convergence rate
\begin{equation}\label{linear1}
  f(x_{k+1})-\bar{f}\leq (1-\frac{\nu}{L})(f(x_{k})-\bar{f}).
\end{equation}
\end{lemma}
The proof consists of two steps. First, using the $L$-smoothness \eqref{sm} with $y=x_{k+1}$ and $x=x_k$, we obtain
\begin{equation}\label{step1}
  f(x_{k+1})\leq f(x_{k}) -\frac{1}{2L}\|\nabla f(x_k)\|^2.
\end{equation}
Second, using the PL inequality with $x=x_k$, we get
\begin{equation}\label{step2}
\frac{1}{2}\|\nabla f(x_k)\|^2\geq \nu (f(x_k)-\bar{f}).
\end{equation}
The linear convergence result \eqref{linear1} follows directly by combining \eqref{step1} and \eqref{step2}. In what follows, we will show that the rate of linear convergence \eqref{linear1} can be improved if the objective $f$ is not only $L$-smooth but also convex. From Corollary \ref{cor2}, we know that \eqref{psm2} characterizes the $L$-smoothness and the convexity of $f$ at the same time. Now, using \eqref{psm2} with $x=x_{k+1}$ and $y=x_k$, we obtain
\begin{equation}\label{step3}
  f(x_{k+1})\leq f(x_{k}) -\frac{1}{2L}\|\nabla f(x_k)\|^2-\frac{1}{2L}\|\nabla f(x_{k+1})\|^2.
\end{equation}
Note that the PL inequality $x=x_{k+1}$ implies that
\begin{equation}\label{step4}
\frac{1}{2}\|\nabla f(x_{k+1})\|^2\geq \nu (f(x_{k+1})-\bar{f}).
\end{equation}
Now, combining \eqref{step3}, \eqref{step2} and \eqref{step4}, we have the following improved linear convergence.
\begin{proposition}
Consider the unconstrained optimization problem $\min_{x\in\RR^d} f(x)$, where $f$ is convex and $L$-smooth, has a nonempty solution set $\cX^*$ with $\bar{f}$ being the optimal function value, and satisfies the PL inequality \eqref{pl}. Then the gradient descent \eqref{GD} with a step-size of $\frac{1}{L}$ has a global linear convergence rate
\begin{equation}\label{linear2}
  f(x_{k+1})-\bar{f}\leq \frac{L-\nu}{L+\nu}(f(x_{k})-\bar{f}).
\end{equation}
\end{proposition}

\section*{Acknowledgements}
The third author was supported by the National Science Foundation of China (No.11971480), the Natural Science Fund of Hunan for Excellent Youth (No.2020JJ3038), and the Fund for NUDT Young Innovator Awards (No. 20190105).

\small

\end{document}